\documentclass[11pt]{article}
\usepackage{amssymb}
\usepackage[usenames,dvipsnames]{xcolor}
\usepackage{amsthm}
\usepackage{amsmath}
\usepackage{amsfonts}

\usepackage{graphicx}

 \newtheorem{thm}{Theorem}[section]
 
 \newtheorem{lem}[thm]{Lemma}
 \newtheorem{prop}[thm]{Proposition}
 \theoremstyle{definition}
 
 \newtheorem{rem}[thm]{Remark}
 \numberwithin{equation}{section}

\newtheorem{theorem}{Theorem}[section]

\theoremstyle{definition}

\theoremstyle{remark}
\newtheorem{remark}[theorem]{Remark}

\begin{document}

\title{On Axially Symmetric Incompressible
Magnetohydrodynamics in Three Dimensions}

\author{Zhen Lei
\footnote{School of Mathematical Sciences; LMNS and Shanghai Key
  Laboratory for Contemporary Applied Mathematics, Fudan
  University, Shanghai 200433, P. R. China. Email: leizhn@gmail.com} }

\date{\today}
\maketitle

\begin{abstract}
In this short article, we prove the global regularity of axially
symmetric solutions to the systems of incompressible ideal magnetohydrodynamics
and resistive magnetohydrodynamics
in three dimensions in the csae that the magnetic fields are purely swirling and perpendicular
to the velocity fields.
\end{abstract}


\section{Introduction}

Magnetohydrodynamics (MHD) is to study the behavior of an
electrically-conducting fluids. Examples of such fluids include
plasmas, liquid metals, salt water, etc. The field of MHD was
initiated by Hannes Alfv${\rm \acute{e}}$n, for which he received
the Nobel Prize in Physics in 1970. However, the mathematical theory
on MHD is still very little known until today.

The fundamental concept behind MHD is that magnetic fields can
induce currents in a moving conductive fluid, which in turn creates
forces on the fluid and also changes the magnetic field itself. MHD
owes its peculiar interest and difficulty to this interaction
between the field and the fluid motion. The set of equations which
describe MHD are a combination of the Navier-Stokes equations of
fluid dynamics and Maxwell's equations of electromagnetism.

Our main result concerns the following incompressible three-dimensional ideal MHD:
\begin{equation}\label{iMHD}
\begin{cases}
\partial_t {\bf u} + ({\bf u} \cdot\nabla) {\bf u} + \nabla
p = \mu\Delta {\bf u} + \frac{1}{\mu_0}(\nabla\times {\bf B})\times {\bf B},\\[-4mm]\\
\partial_t{\bf B}  =  \nabla\times ({\bf u} \times {\bf B}),\\[-4mm]\\
\nabla\cdot {\bf u} =0,\quad \nabla\cdot {\bf B}=0,
\end{cases}
\end{equation}
where ${\bf B}$ denotes the magnetic field,  ${\bf u}$ the bulk
plasma velocity and $p$ the plasma pressure. The magnetic constant $\mu_0$ and the fluid viscosity $\mu$ are both positive.
We will set all the constants to be 1 since they play no role in this paper. The ideal
MHD is used when the electrically-conducting fluid has so little resistivity that it can be treated as a
perfect conductor. This is the limit of infinite magnetic Reynolds
number. For applications of ideal MHD, see, for instance, \cite{CC}.

The following theorem shows that if the magnetic field is purely swirling
and is perpendicular  to the velocity field, then the 3D
incompressible ideal MHD \eqref{iMHD} is globally well-posed in the axially symmetric case.
\begin{thm}\label{thm-iMHD}
Suppose that ${\bf u}_0$ and ${\bf B}_0$ are both axially symmetric
divergence-free vectors with $u_0^\theta = 0$ and $B^r_0 = B^z_0 =
0$. Moreover, we assume that $({\bf u}_0, {\bf B}_0) \in H^2$ and $\frac{B^\theta_0}{r} \in L^\infty$. Then there
exists a unique global solution $({\bf u}, {\bf B})$ for the ideal
MHD \eqref{iMHD} with the initial data $({\bf u}_0, {\bf B}_0)$
which satisfies
\begin{equation}\nonumber
\|{\bf u}(t, \cdot)\|_{H^2}^2 + \|{\bf B}(t, \cdot)\|_{H^2}^2 +
\int_0^t\|\nabla{\bf u}\|_{H^2}^2ds \lesssim
e^{e^{e^{t^{\frac{5}{4}}}}}.
\end{equation}
\end{thm}

The notations used here will be introduced in section 2. Note that the
Faraday's equation for ${\bf B}$ in \eqref{iMHD} is exactly the same
as the vorticity equation for the 3D incompressible Euler equations
(by identifying ${\bf B}$ and $\nabla \times {\bf u}$). This may
lead to an essential difficulties for the global well-posedness of
the ideal MHD \eqref{iMHD} in general case. Indeed, the global
regularity of \eqref{iMHD} is  widely open in the
even two-dimensional case if the magnetic field is non-trivial. We
achieve Theorem \ref{iMHD} by exploring the underlying special
structures of the MHD system in axially symmetric case. The magnetic
stretching term ${\bf B}\cdot\nabla {\bf u}$ in Faraday's equation
can be absorbed into the convection term by dividing the equation by
$r$. This yields that $\Pi = \frac{B^\theta}{r}$ is only transported
by the velocity field ${\bf u}$. On the other hand, by dividing $r$
in the vorticity equation, one can absorb the vortex stretching term
into the convection term, leaving only one term involving $\Pi$ as a
forcing one in $\Omega$ equation. See section 2 and 3 for more details.

Similar result in Theorem \ref{iMHD} is of course expected to hold for the
following resistive MHD:
\begin{equation}\label{RMHD}
\begin{cases}
\partial_t {\bf u} + ({\bf u} \cdot\nabla) {\bf u} + \nabla
p = \Delta {\bf u} + (\nabla\times {\bf B})\times {\bf B},\\[-4mm]\\
\partial_t{\bf B}  =  \nu\Delta {\bf B} + \nabla\times ({\bf u} \times {\bf B}),\\[-4mm]\\
\nabla\cdot {\bf u} =0,\quad \nabla\cdot {\bf B}=0.
\end{cases}
\end{equation}
Again, we will set the resistivity constant $\nu > 0$  to be 1
since it plays no role here. We have the following theorem:
\begin{thm}\label{thm-RMHD}
Suppose that ${\bf u}_0$ and ${\bf B}_0$ are both axially symmetric
divergence-free vectors with $u_0^\theta = 0$ and $B^r_0 = B^z_0 =
0$. Moreover, we assume that $({\bf u}_0, {\bf B}_0) \in H^1$ and
$\frac{B^\theta_0}{r} \in L^\infty$. Then there exists a unique
global solution $({\bf u}, {\bf B})$ for the resistive MHD
\eqref{RMHD} with the initial data $({\bf u}_0, {\bf B}_0)$.
Moreover, $({\bf u}, {\bf B})$ is smooth in the sense that $({\bf
u}(t, \cdot), {\bf B}(t, \cdot)) \in H^s$ for any $s \geq 0$ and $t
> 0$.
\end{thm}

Our motivation of the above results is a novelty observation on the connections between
MHD and axially symmetric Navier-Stokes equations.
If we rewrite the 3D incompressible
axially symmetric Navier-Stokes equations as \eqref{NS-axi}, in terms of ${\bf u} =
u^r{\bf e}_r + u^z{\bf e}_z$ and ${\bf b} = u^\theta {\bf
e}_\theta$, then there is only a sign difference\footnote{In fact,
the pressure is also changed. But the pressure is not a troublesome
term for our purpose due to the divergence-free condition.} between
the Navier-Stokes equations \eqref{NS-axi} for $({\bf u}, {\bf b})$
and the resistive MHD \eqref{RMHD} for $({\bf u}, {\bf B})$ (see
Remark \ref{connection} in section 2 for details). However, this
difference of sign significantly changes the difficulties in solving
3D axially symmetric incompressible equations of MHD.

We remark that the perfect resistive case will be treated in a forthcoming paper \cite{LeiLi}. It is also
interesting to consider the case when $u^\theta = B^\theta = 0$.

Before ending the introduction, let us mention some important
results in the field of incompressible MHD. The local well-posedness
of the resistive MHD \eqref{RMHD} was established in
\cite{SermangeTemam} where the authors also proved the global
well-posedness in 2D case. A nontrivial blowup criterion for the
perfect resistive MHD was established in terms of only
$L^1_t({\rm BMO})$ norm of vorticity of the velocity field in
\cite{LeiZ-2}. Recently, Lin, Xu and Zhang \cite{LXZ} obtained the global well-posedness of classical
solutions for the 2D ideal MHD \eqref{iMHD} under the assumption
that the initial velocity field and the displacement of the magnetic
field from a non-zero constant is sufficiently small in appropriate
Sobolev spaces. Cao and Wu \cite{CWu} proved the global regularity
of 2D resistive MHD with partial viscosity and resistivity (see also \cite{CWY} and the references therein). We also
emphasis the partial regularity theory and Serrin type criterions in
\cite{HeXin1, HeXin2}, and various blowup criterions in \cite{CKS,
CMZ} (see also the reference therein).

The remaining of this paper is simply organized as follows: In
section 2 we will derive the axisymmetric MHD in cylindrical
coordinate. We will also make a comment on the difference between
the resistive MHD \eqref{RMHD} and the axially symmetric
Navier-Stokes equations and prove a maximum principle for $\Pi$. We
will prove Theorem \ref{thm-iMHD} in section 3. Then in section 4 we present the proof of Theorem \ref{RMHD}.

\section{Axially Symmetric MHD and A Maximum Principle}

In this section we will first derive the incompressible axially
symmetric MHD in cylindrical coordinate. Then we show that the
quantity $\Pi$ satisfies a maximum principle. We also present an
interesting connection between the axisymmetric MHD studied in
Theorem \ref{thm-RMHD} and the axisymmetric Navier-Stokes equations
with non-trivial swirl $u^\theta$ (see \eqref{mhd-axi} and
\eqref{NS-axi}).

Let us begin with some notations. A point in $\mathbb{R}^3$ is
denoted by ${\bf x} = (x_1, x_2, z)$. Let $r = \sqrt{x_1^2 + x_2^2}$
and
\begin{equation}\nonumber
{\bf e}_r = \begin{pmatrix}\frac{x_1}{r}\\ \frac{x_2}{r}\\ 0
  \end{pmatrix},
\quad {\bf e}_\theta = \begin{pmatrix}- \frac{x_2}{r}\\ \frac{x_1}{r}\\
0
  \end{pmatrix},
\quad {\bf e}_z = \begin{pmatrix}0\\ 0\\ 1
  \end{pmatrix}
\end{equation}
be the three orthogonal unit vectors along the radial, the angular,
and the axial directions respectively. An axially symmetric solution
to the 3D incompressible MHD \eqref{RMHD} is a solution $({\bf u},
{\bf B}, p)$ which takes the following form
\begin{equation}\nonumber
\begin{cases}
{\bf u}(t, {\bf x}) = u^r(t, r, z){\bf e}_r +  u^\theta(t, r, z){\bf
  e}_\theta + u^z(t, r, z){\bf e}_z,\\[-4mm]\\
{\bf B}(t, {\bf x}) = B^r(t, r, z){\bf e}_r +  B^\theta(t, r, z){\bf
  e}_\theta + B^z(t, r, z){\bf e}_z,\\[-4mm]\\
p(t, {\bf x}) = p(t, r, z).
\end{cases}
\end{equation}
We will also write the vorticity field $\nabla\times {\bf u}$ in
cylindrical coordinate:
\begin{equation}\nonumber
\nabla\times {\bf u}(t, {\bf x}) = \omega^r(t, r, z){\bf e}_r +
\omega^\theta(t, r, z){\bf e}_\theta + \omega^z(t, r, z){\bf e}_z,
\end{equation}
where
\begin{equation}\nonumber
\omega^r = - \partial_zu^\theta,\quad \omega^\theta = \partial_zu^r
- \partial_ru^z,\quad \omega^z = \frac{1}{r}\partial_r(ru^\theta).
\end{equation}
Define
\begin{equation}\label{defn-1}
\Pi = \frac{B^\theta}{r},\quad \Omega =
\frac{\omega^\theta}{r},\quad \Gamma = ru^\theta.
\end{equation}

By expanding the Lorentz force term as
\begin{equation}\nonumber
(\nabla\times {\bf B}) \times {\bf B} = {\bf B}\cdot\nabla {\bf B} -
\nabla\frac{|{\bf B}|^2}{2},
\end{equation}
and then taking the inner product of ${\bf u}$ and ${\bf B}$
equations with ${\bf e_r}$, ${\bf e_\theta}$ and ${\bf e_z}$,
respectively, we can derive the resistive MHD in cylindrical
coordinate:
\begin{equation}\label{mhd-axi-g}
\begin{cases}
\partial_tu^r + u^r\partial_ru^r + u^z\partial_zu^r - \frac{(u^\theta)^2}{r}
  + \partial_rP\\
\quad\quad\quad = \big(\Delta - \frac{1}{r^2}\big)u^r
  + B^r\partial_rB^r + B^z\partial_zB^r - \frac{(B^\theta)^2}{r},\\[-4mm]\\
\partial_tu^\theta + u^r\partial_ru^\theta + u^z\partial_zu^\theta + \frac{u^ru^\theta}{r}\\
\quad\quad\quad  = \big(\Delta - \frac{1}{r^2}\big)u^\theta
  + B^r\partial_rB^\theta + B^z\partial_zB^\theta + \frac{B^rB^\theta}{r},\\[-4mm]\\
\partial_tu^z + u^r\partial_ru^z + u^z\partial_zu^z + \partial_zP\\
\quad\quad\quad\quad\quad = \Delta u^z  + B^r\partial_rB^z + B^z\partial_zB^z,\\[-4mm]\\
\partial_tB^r + u^r\partial_rB^r + u^z\partial_zB^r\\
\quad\quad\quad  = \big(\Delta - \frac{1}{r^2}\big)B^r
  + B^r\partial_ru^r + B^z\partial_zu^r,\\[-4mm]\\
\partial_tB^\theta + u^r\partial_rB^\theta + u^z\partial_zB^\theta + \frac{B^ru^\theta}{r}\\
\quad\quad\quad  = \big(\Delta - \frac{1}{r^2}\big)B^\theta
  + B^r\partial_ru^\theta + B^z\partial_zu^\theta + \frac{u^rB^\theta}{r},\\[-4mm]\\
\partial_tB^z + u^r\partial_rB^z + u^z\partial_zB^z\\
\quad\quad\quad = \Delta B^z  + B^r\partial_ru^z + B^z\partial_zu^z,
\end{cases}
\end{equation}
where the pressure is given by
\begin{equation}\label{p}
P = p + \frac{|{\bf B}|^2}{2}.
\end{equation}
The incompressible
constraints are
\begin{equation}\label{incom-1}
\partial_ru^r + \frac{u^r }{r} + \partial_zu^z = 0,\quad
\partial_rB^r + \frac{B^r }{r} + \partial_zB^z =
0.
\end{equation}

The general axially symmetric resistive MHD is governed by
\eqref{mhd-axi-g} and \eqref{incom-1}. In this paper, we consider a
family of solutions with the form
\begin{equation}\label{axi-solu}
{\bf u}(t, {\bf x}) = u^r(t, r, z){\bf e}_r + u^z(t, r, z){\bf
  e}_z,\quad {\bf B}(t, {\bf x}) = B^\theta(t, r, z){\bf e}_\theta.
\end{equation}
It is easy to check that $(u^\theta, B^r, B^z)$ can be zero for all
time if they are zero initially. In this case, $({\bf u}, {\bf B},
P)$ in \eqref{axi-solu} and \eqref{p} is governed by
\begin{equation}\label{mhd-axi}
\begin{cases}
\partial_tu^r + u^r\partial_ru^r + u^z\partial_zu^r
  + \partial_rP = \big(\Delta - \frac{1}{r^2}\big)u^r
  - \frac{(B^\theta)^2}{r},\\[-4mm]\\
\partial_tu^z + u^r\partial_ru^z + u^z\partial_zu^z + \partial_zP  = \Delta u^z,\\[-4mm]\\
\partial_tB^\theta + u^r\partial_rB^\theta + u^z\partial_zB^\theta
  = \big(\Delta - \frac{1}{r^2}\big)B^\theta
  + \frac{u^rB^\theta}{r},
\end{cases}
\end{equation}
together with the incompressible constraint
\begin{equation}\label{incom}
\partial_ru^r + \frac{u^r }{r} + \partial_zu^z = 0.
\end{equation}
To avoid the explicit presence of pressure, we also need the
vorticity formula of \eqref{mhd-axi}:
\begin{equation}\label{vorticity-axi}
\begin{cases}
\partial_tB^\theta + u^r\partial_rB^\theta + u^z\partial_zB^\theta
  = \big(\Delta - \frac{1}{r^2}\big)B^\theta
  + \frac{u^rB^\theta}{r},\\[-4mm]\\
\partial_t\omega^\theta + u^r\partial_r\omega^\theta + u^z\partial_z\omega^\theta
  - \frac{u^r\omega^\theta}{r} = \big(\Delta - \frac{1}{r^2}\big)
  \omega^\theta - \frac{\partial_z(B^\theta)^2}{r}.
\end{cases}
\end{equation}

\begin{remark}\label{connection}
It is well-known that the axially symmetric Navier-Stokes equations
(in the case of ${\bf B} \equiv 0$) are (see, for instance,
\cite{MB})
\begin{equation}\nonumber
\begin{cases}
\partial_tu^r + u^r\partial_ru^r + u^z\partial_zu^r
  + \partial_rp = \big(\Delta - \frac{1}{r^2}\big)u^r
  + \frac{(u^\theta)^2}{r},\\[-4mm]\\
\partial_tu^z + u^r\partial_ru^z + u^z\partial_zu^z + \partial_zp  = \Delta u^z,\\[-4mm]\\
\partial_tu^\theta + u^r\partial_ru^\theta + u^z\partial_zu^\theta
  = \big(\Delta - \frac{1}{r^2}\big)u^\theta
  - \frac{u^ru^\theta}{r}.
\end{cases}
\end{equation}
If we denote ${\bf u} = u^r{\bf e}_r + u^z{\bf e}_z$ and ${\bf b} =
u^\theta {\bf e}_\theta$, we can rewrite the above axially symmetric
Navier-Stokes equations as
\begin{equation}\label{NS-axi}
\begin{cases}
\partial_t{\bf u} + {\bf u}\cdot\nabla {\bf u}
  + \nabla p = \Delta{\bf u}
  - {\bf b}\cdot\nabla {\bf b},\\[-4mm]\\
\partial_t{\bf b} + {\bf u}\cdot\nabla{\bf b}  = \Delta {\bf b} - {\bf b}\cdot\nabla{\bf u},\\[-4mm]\\
\nabla\cdot{\bf u} = \nabla\cdot{\bf b} = 0.
\end{cases}
\end{equation}
If we compare the MHD equations \eqref{RMHD} with the Navier-Stokes
equations \eqref{NS-axi}, we find that we can recover \eqref{RMHD}
from \eqref{NS-axi} by changing the sign of the terms ${\bf
b}\cdot\nabla {\bf b}$ and ${\bf b}\cdot\nabla {\bf u}$.   The
significance of the tiny difference, especially, the sign of ${\bf
b}\cdot\nabla {\bf u}$, yields a much stronger \textit{a priori}
estimate in the MHD case.
\end{remark}

\begin{prop}[Maximum Principle]\label{prop-maxi}
Assume that $({\bf u}, {\bf B}, P)$ is a smooth bounded solution to
\eqref{mhd-axi} with or without resistivity. Then the quantity $\Pi$ satisfies the maximum principle
\begin{equation}\nonumber
\|\Pi(t, \cdot)\|_{L^\infty} \leq \|\Pi(0, \cdot)\|_{L^\infty},\quad
\forall\ t \geq 0.
\end{equation}
\end{prop}
\begin{proof}
In the case of zero resistivity, by dividing the
equation for $B^\theta$ by $r$, one has
\begin{equation}\label{Pi-eqn-1}
\partial_t\Pi + u^r\partial_r\Pi + u^z\partial_z\Pi
  = 0,
\end{equation}
which gives the maximum principle for $\Pi$ in the case of zero
resistivity.

Similarly, in the resistive case, we have
\begin{equation}\label{Pi-eqn-2}
\partial_t\Pi + u^r\partial_r\Pi + u^z\partial_z\Pi
  = (\partial_r^2 + \frac{3}{r}\partial_r + \partial_z^2)\Pi.
\end{equation}
Then the maximum principle follows by interpreting $(\partial_r^2 + \frac{3}{r}\partial_r + \partial_z^2)$
 as a five-dimensional Laplacian operator.

\end{proof}

\section{Proof of Theorem \ref{thm-iMHD}}

In this section we prove Theorem \ref{thm-iMHD}. Throughout this
paper, we will use $A_1 \lesssim A_2$ to denote that $A_1 \leq
C_0A_2$ and $A_1 \simeq A_2$ to denote that $C_0^{-1}A_2 \leq A_1
\leq C_0A_2$ for a generic positive constant $C_0
> 1$ and two positive quantities $A_1$ and $A_2$.

\begin{proof}[Proof of Theorem \ref{thm-iMHD}]
Let us rewrite the vorticity equation in \eqref{vorticity-axi} in
terms of $\Omega$:
\begin{equation}\nonumber
\partial_t\Omega + u^r\partial_r\Omega + u^z\partial_z\Omega
= (\partial_r^2 + \frac{3}{r}\partial_r + \partial_z^2)\Omega -
\partial_z\Pi^2.
\end{equation}

By taking the $L^2$ inner product of the above equation with
$\Omega$ and preforming the standard energy estimate, one has
\begin{eqnarray}\nonumber
&&\frac{1}{2}\frac{d}{dt}\|\Omega\|_{L^2}^2 -
  \int\Omega(\partial_r^2 + \frac{3}{r}\partial_r +
  \partial_z^2)\Omega dx\\\nonumber
&&= - \frac{1}{2}\int (u^r\partial_r\Omega^2 + u^z\partial_z
  \Omega^2)dx - \int\Omega\partial_z\Pi^2dx.
\end{eqnarray}
Using the incompressibility condition \eqref{incom} and the fact of
$dx = 2\pi rdrdz$, one has
\begin{eqnarray}\nonumber
\int (u^r\partial_r\Omega^2 + u^z\partial_z\Omega^2) dx = 0
\end{eqnarray}
and
\begin{eqnarray}\nonumber
- \int\Omega(\partial_r^2 + \frac{3}{r}\partial_r +
  \partial_z^2)\Omega dx = \|\nabla\Omega\|_{L^2}^2 + 2\pi\int_{\mathbb{R}}|\Omega(t, 0, z)|^2dz.
\end{eqnarray}
By integration by part and interpolation, we have
\begin{eqnarray}\nonumber
\big|\int\Omega\partial_z\Pi^2dx\big| \leq
\|\Pi\|_{L^4}^2\|\partial_z\Omega\|_{L^2} \leq
\frac{1}{2}\|\Pi\|_{L^2}^2\|\Pi\|_{L^\infty}^2 +
\frac{1}{2}\|\partial_z\Omega\|_{L^2}^2.
\end{eqnarray}
Consequently, one has
\begin{eqnarray}\label{3-1}
\frac{d}{dt}\|\Omega\|_{L^2}^2 + \|\nabla\Omega \|_{L^2}^2 \leq
\|\Pi\|_{L^2}^2\|\Pi\|_{L^\infty}^2.
\end{eqnarray}

Similarly, using equation \eqref{Pi-eqn-2} and preforming the $L^2$
energy estimate, one has
\begin{eqnarray}\label{3-2}
\|\Pi(t, \cdot)\|_{L^2} \leq \|\Pi_0\|_{L^2},\quad \forall\ t \geq
0.
\end{eqnarray}
Consequently, by Proposition \ref{prop-maxi} and using \eqref{3-1},
\eqref{3-2}, we have
\begin{equation}\label{3-3}
\|\Omega(t, \cdot)\|_{L^2} \lesssim 1 + \sqrt{t},\quad
\int_0^t\|\nabla\Omega\|_{L^2}^2dt \lesssim 1 + t, \quad \forall\ t
\geq 0.
\end{equation}
Here we used that $\Omega_0 \in L^2$ which is due to the fact that
${\bf u}_0 \in H^2$ and
\begin{equation}\nonumber
\big|\nabla(\nabla\times {\bf u})\big|^2 = \big|({\bf e}_r\partial_r
+ \frac{1}{r}{\bf e}_\theta\partial_\theta + {\bf
e}_z\partial_z)\omega^\theta {\bf e}_\theta\big|^2 =
|\nabla\omega^\theta|^2 + |\Omega|^2.
\end{equation}
Similarly, one also has $\Pi_0 \in L^2$ since ${\bf B}_0 \in H^1$.

To proceed, we need a technical lemma regarding the property of a
Riesz operator on $\mathbb{R}^3$. We first recall the following
weighted Calderon-Zygmund inequality for a singular integral
operator with a weight function which is in the $\mathcal{A}_p$
class (see Stein \cite{Stein} pp. 194-217 for details). Let
$\mathcal{K}$ be a Riesz operator in $\mathbb{R}^n$ and $w(x)$ be a
weight in the $\mathcal{A}_p$ class (see page 194 of \cite{Stein}
for definition). One can extend the Calderon-Zygmund inequality for
the singular integral operator with the integral having weight
function $w(x)$. Specifically, for $1 < p < \infty$, there holds
\begin{equation}\nonumber
\|\mathcal{K}f\|_{L^p(\mathbb{R}^n)} \lesssim
\|f\|_{L^p(\mathbb{R}^n)},\quad \forall\ \ f \in L^p(\mathbb{R}^n).
\end{equation}
The following lemma plays an essential role in our global regularity
analysis.
\begin{lem}\label{C-Z}
There holds
\begin{equation}\nonumber
\int_0^T\|r^{-1}u^r(t, \cdot)\|_{L^\infty}dt \lesssim \sup_{0 \leq t
\leq T}\|\Omega(t,
\cdot)\|_{L^2}^{\frac{1}{2}}\int_0^T\|\nabla\Omega(t,
\cdot)\|_{L^2}^{\frac{1}{2}}dt.
\end{equation}
\end{lem}
\begin{rem}\nonumber
We pointed out that in \cite{HLL} the authors have established an
inequality $\|r^{-1}\partial_zu^r\|_{L^p} \lesssim \|\Omega\|_{L^p}$
for $1 < p < \infty$ for $\mathbb{R}^2 \times \mathbb{T}^1$, where
$\mathbb{T}^1$ is a one-dimensional torus.
\end{rem}
\begin{proof}
We follow the proof in \cite{HLL}. By the incompressible constraint
\eqref{incom}, we can introduce the angular stream function
$\psi^\theta$ such that
\begin{equation}\label{3-4}
- \big(\partial_r^2 + \frac{1}{r}\partial_r - \frac{1}{r^2} +
\partial_z^2\big)\psi^\theta = \omega^\theta,
\end{equation}
and
\begin{equation}\nonumber
u^r = - \partial_z\psi^\theta,\quad u^z =
\frac{1}{r}\partial_r(r\psi^\theta).
\end{equation}
We divide by $r$ in \eqref{3-4}, which gives that
\begin{equation}\label{3-5}
- \big(\partial_r^2 + \frac{3}{r}\partial_r +
\partial_z^2\big)\frac{\psi^\theta}{r} = \frac{\omega^\theta}{r}.
\end{equation}

Following \cite{HLL}, we interpret the Laplace operator in
\eqref{3-5} as a five-dimensional one. We formally write
$$y = (y_1, y_2, y_3, y_4, z),\quad r = \sqrt{y_1^2 + y_2^2 + y_3^2 + y_4^2},
\quad \Delta_y = \big(\partial_r^2 + \frac{3}{r}\partial_r +
\partial_z^2\big).$$
This way we have $\frac{\psi^\theta}{r} =
(-\Delta_y)^{-1}\frac{\omega^\theta}{r}.$ In the remaining part of
the proof of this lemma, we will use a subscript $y$ to denote the
derivatives with respect to $y$.

It is clear that
\begin{eqnarray}\nonumber
\nabla^2\frac{\psi^\theta}{r} &=& ({\bf e}_r\partial_r +
  \frac{1}{r}{\bf e}_\theta\partial_\theta +
  {\bf e}_z\partial_z)\big({\bf e}_r\partial_r\frac{\psi^\theta}{r}\big) +
  \nabla\partial_z\frac{\psi^\theta}{r}\otimes {\bf e}_z\\\nonumber
&=& {\bf e}_r\otimes {\bf e}_r\partial_r^2\frac{\psi^\theta}{r} +
  {\bf e}_\theta\otimes {\bf e}_\theta\frac{1}{r}\partial_r\frac{\psi^\theta}{r}
  + ({\bf e}_z\otimes {\bf e}_r + {\bf e}_r\otimes{\bf e}_z)\partial_{zr}^2\frac{\psi^\theta}{r} +
  {\bf e}_z\otimes {\bf e}_z\partial_z^2\frac{\psi^\theta}{r}.
\end{eqnarray}
Consequently, one has
\begin{eqnarray}\label{S7-1}
\big|\nabla^2\frac{\psi^\theta}{r}\big|^2 \simeq
\big|\partial_r^2\frac{\psi^\theta}{r}\big|^2 +
\big|\frac{1}{r}\partial_r\frac{\psi^\theta}{r}\big|^2 +
\big|\partial_z^2\frac{\psi^\theta}{r}\big|^2 +
\big|\partial_{rz}^2\frac{\psi^\theta}{r}\big|^2.
\end{eqnarray}
On the other hand, one also has
\begin{eqnarray}\nonumber
\nabla_y^2\frac{\psi^\theta}{r} &=& (\widetilde{{\bf e}}_r\partial_r
  + \nabla_\theta + \widetilde{{\bf e}}_z\partial_z)\big(\widetilde{{\bf e}}_r
  \partial_r\frac{\psi^\theta}{r}\big) +
  \nabla_y\partial_z\frac{\psi^\theta}{r}\otimes \widetilde{{\bf e}}_z\\\nonumber
&=& \widetilde{{\bf e}}_r\otimes \widetilde{{\bf
  e}}_r\partial_r^2\frac{\psi^\theta}{r} +
  \nabla_\theta \widetilde{{\bf e}}_r\partial_r\frac{\psi^\theta}{r}
  + (\widetilde{{\bf e}}_z\otimes \widetilde{{\bf e}}_r + \widetilde{{\bf e}}_r
  \otimes \widetilde{{\bf e}}_z)\partial_{zr}^2\frac{\psi^\theta}{r} +
  \widetilde{{\bf e}}_z\otimes \widetilde{{\bf e}}_z\partial_z^2\frac{\psi^\theta}{r}\\\nonumber
&=& \widetilde{{\bf e}}_r\otimes \widetilde{{\bf e}}_r
  \partial_r^2\frac{\psi^\theta}{r} + (I_0 - \widetilde{{\bf e}}_r\otimes
  \widetilde{{\bf
  e}}_r)\frac{1}{r}\partial_r\frac{\psi^\theta}{r}\\\nonumber
&&  +\ (\widetilde{{\bf e}}_z\otimes \widetilde{{\bf e}}_r +
\widetilde{{\bf e}}_r
  \otimes \widetilde{{\bf e}}_z) \partial_{zr}^2\frac{\psi^\theta}{r} +
  \widetilde{{\bf e}}_z\otimes \widetilde{{\bf e}}_z\partial_z^2\frac{\psi^\theta}{r}.
\end{eqnarray}
where $I_0 = \begin{pmatrix}I_{4\times 4} & 0 \\ 0 & 0\end{pmatrix}$
and $\nabla_\theta$ is defined by
\begin{equation}\nonumber
\nabla_\theta = \nabla - \widetilde{{\bf e}}_r(\widetilde{{\bf
e}}_r\cdot\nabla_y) - \widetilde{{\bf e}}_z\partial_z,\quad
\widetilde{{\bf e}}_r = \frac{1}{r}\begin{pmatrix}y_1\\ y_2\\ y_3\\ y_4\\
0\end{pmatrix},
\widetilde{{\bf e}}_z = \begin{pmatrix}0\\ 0\\ 0\\ 0\\
1\end{pmatrix}.
\end{equation}
Clearly, $\widetilde{{\bf e}}_r\otimes \widetilde{{\bf e}}_r$, $I_0
- \widetilde{{\bf e}}_r\otimes \widetilde{{\bf e}}_r$,
$\widetilde{{\bf e}}_z\otimes \widetilde{{\bf e}}_r$,
$\widetilde{{\bf e}}_r\otimes \widetilde{{\bf e}}_z$ and
$\widetilde{{\bf e}}_z\otimes \widetilde{{\bf e}}_z$ are all
mutually orthogonal. Consequently, one also has
\begin{equation}\label{S7-2}
\big|\nabla_y^2\frac{\psi^\theta}{r}\big|^2 \simeq
\big|\partial_r^2\frac{\psi^\theta}{r}\big|^2 +
\big|\frac{1}{r}\partial_r\frac{\psi^\theta}{r}\big|^2 +
\big|\partial_z^2\frac{\psi^\theta}{r}\big|^2 +
\big|\partial_{rz}^2\frac{\psi^\theta}{r}\big|^2.
\end{equation}
By \eqref{S7-1} and \eqref{S7-2}, we have
\begin{eqnarray}\nonumber
\int\big|\nabla^2\frac{\psi^\theta}{r}\big|^pdx &\simeq&
  \int_{-\infty}^\infty\int_0^\infty\Big(\big|\partial_r^2
  \frac{\psi^\theta}{r}\big|^2 + \big|\frac{1}{r}\partial_r\frac{\psi^\theta
  }{r}\big|^2 + \big|\partial_z^2\frac{\psi^\theta}{r}\big|^2 +
  \big|\partial_{rz}^2\frac{\psi^\theta}{r}\big|^2\Big)^{\frac{p}{2}}rdrdz\\\nonumber
&=& \int_{-\infty}^\infty\int_0^\infty\Big(\big|\partial_r^2
  \frac{\psi^\theta}{r}\big|^2 + \big|\frac{1}{r}\partial_r\frac{\psi^\theta
  }{r}\big|^2 + \big|\partial_z^2\frac{\psi^\theta}{r}\big|^2 +
  \big|\partial_{rz}^2\frac{\psi^\theta}{r}\big|^2\Big)^{\frac{p}{2}}w(r)r^3drdz\\\nonumber
&\simeq& \int_{-\infty}^\infty\int_0^\infty\big|\nabla_y^2
  \frac{\psi^\theta}{r}\big|^pw(r)r^3drdz\\\nonumber
&\simeq& \int\big|\nabla_y^2(-\Delta_y)^{-1}
  \frac{\omega^\theta}{r}\big|^pw(r)dy,
\end{eqnarray}
where $w(r)$ is a weight function $w(r) = \frac{1}{r^2}$.

Let $1 < p < \infty$. Using Lemma 2 in \cite{HLL} (see also a
general version in Lemma \ref{A-p} in Appendix of this paper), we
have
\begin{eqnarray}\nonumber
\int\big|\nabla_y^2(-\Delta_y)^{-1}
  \frac{\omega^\theta}{r}\big|^pw(r)dy &\lesssim& \int\big|
  \frac{\omega^\theta}{r}\big|^pw(r)dy\\\nonumber
&\simeq& \int\big|\frac{\omega^\theta}{r}\big|^pdx.
\end{eqnarray}
Consequently, one has
\begin{equation}\label{3-6}
\int\big|\nabla^2\frac{\psi^\theta}{r}\big|^pdx \lesssim
\int\big|\frac{\omega^\theta}{r}\big|^pdx.
\end{equation}
Repeating the above procedure, one also has
\begin{equation}\label{3-7}
\int\big|\nabla^2\frac{\partial_z\psi^\theta}{r}\big|^pdx \lesssim
\int\big|\frac{\partial_z\omega^\theta}{r}\big|^pdx.
\end{equation}

Taking $p = 2$ in \eqref{3-6} and \eqref{3-7} and using the
interpolation inequality $\|f\|_{L^\infty} \lesssim \|\nabla
f\|_{L^2}^{\frac{1}{2}}\|\nabla^2f\|_{L^2}^{\frac{1}{2}}$ in
$\mathbb{R}^3$, one has
\begin{eqnarray}\nonumber
\int_0^T\big\|r^{-1}u^r(t, \cdot)\big\|_{L^\infty}dt &=&
  \int_0^T\big\|r^{-1}\partial_z\psi^\theta(t,
  \cdot)\big\|_{L^\infty}dt\\\nonumber
&\lesssim& \int_0^T\big\|\nabla\partial_z(r^{-1}\psi^\theta(t,
  \cdot))\big\|_{L^2}^{\frac{1}{2}}\big\|\nabla^2\partial_z(r^{-1}\psi^\theta(t,
  \cdot))\big\|_{L^2}^{\frac{1}{2}}dt\\\nonumber
&\lesssim& \sup_{0 \leq t \leq T}\|\Omega(t,
  \cdot))\|_{L^2}^{\frac{1}{2}}\int_0^T\|\partial_z\Omega(t,
  \cdot)\|_{L^2}^{\frac{1}{2}}dt.
\end{eqnarray}
This finishes the proof of the lemma.
\end{proof}

Now we derive an $L^\infty$ estimate for $B^\theta$. Ignoring the
viscosity in the equation of $B^\theta$ in \eqref{mhd-axi}, one has
\begin{eqnarray}\nonumber
\|B^\theta(t, \cdot)\|_{L^\infty} &\leq& \|B^\theta_0\|_{L^\infty} +
  \int_0^t\|B^\theta(s, \cdot)\|_{L^\infty}\big\|\frac{u^r}{r}\big\|_{L^\infty}ds.
\end{eqnarray}
By Gronwall's inequality and using \eqref{3-3} and Lemma \ref{C-Z},
we have
\begin{eqnarray}\label{9}
\|B^\theta(t, \cdot)\|_{L^\infty} \leq
\|B^\theta_0\|_{L^\infty}e^{\int_0^t\|r^{-1}u^r(s,
\cdot)\|_{L^\infty}ds} \lesssim e^{t^{\frac{5}{4}}}.
\end{eqnarray}

Let us coming back to \eqref{vorticity-axi} and estimate that
\begin{eqnarray}\nonumber
&&\frac{1}{2}\frac{d}{dt}\int|\omega^\theta|^2dx +
  \int\big(|\nabla \omega^\theta|^2 +
  \frac{|\omega^\theta|^2}{r^2}\big)dx\\\nonumber
&&\leq \big\|\frac{u^r}{r}\big\|_{L^\infty}\int(\omega^\theta)^2dx
  + \|\Pi\|_{L^\infty}\|B^\theta\|_{L^2}\|\partial_z\omega^\theta\|_{L^2}\\\nonumber
&&\leq \big\|\frac{u^r}{r}\big\|_{L^\infty}\int(\omega^\theta)^2dx
  + \frac{1}{2}\|\Pi\|_{L^\infty}^2\|B^\theta\|_{L^2}^2
  + \frac{1}{2}\|\partial_z\omega^\theta\|_{L^2}^2.
\end{eqnarray}
Recalling the following basic energy law
\begin{equation}\label{EnergyL}
\frac{1}{2}\frac{d}{dt}\big(\|{\bf u}\|_{L^2}^2 + \|{\bf
B}\|_{L^2}^2\big) + \int_0^t\big\|\nabla{\bf u}\|_{L^2}^2ds = 0,
\end{equation}
and using the \textit{a priori} estimate in \eqref{3-3} and Lemma
\ref{C-Z}, one has
\begin{equation}\label{10}
\|\nabla\times {\bf u}(t, \cdot)\|_{L^2} \lesssim
e^{t^{\frac{5}{4}}},\quad \int_0^t\|\nabla(\nabla\times {\bf
u})\|_{L^2}^2dt \lesssim e^{t^{\frac{5}{4}}}, \quad \forall\ t \geq
0.
\end{equation}

The next step is to bootstrap the regularity of ${\bf u}$ and ${\bf
B}$. We are going to show the $L^1([0, T], {\rm Lip}(\mathbb{R}^3))$
estimate of ${\bf u}$. We will make use of the structure of the
ideal MHD in \eqref{mhd-axi} to avoid some possible technical
complications. The key observation is that we can write the
vorticity equation as
\begin{equation}\nonumber
\partial_t(\nabla\times {\bf u}) + \nabla\times[(\nabla\times {\bf u})\times
  {\bf u}] = \Delta(\nabla\times {\bf u})
  - \partial_z(\Pi B^\theta e_\theta).
\end{equation}
Here by the maximum principle in Proposition \ref{prop-maxi} and
\eqref{9}, one has $\Pi B^\theta \in L^\infty([0, t],
L^\infty(\mathbb{R}^3))$. Moreover, we can apply \eqref{10} to
bootstrap the regularity of $(\nabla\times {\bf u})\times {\bf u}$.
Then we may apply the standard parabolic estimate to get the
$L^1([0, t], L^\infty(\mathbb{R}^3))$ estimate for
$\nabla\times {\bf u}$.

We first perform $L^4$ energy estimate for \eqref{vorticity-axi} and derive
that
\begin{eqnarray}\nonumber
&&\frac{1}{4}\frac{d}{dt}\int|\omega^\theta|^4dx +
  \int\big(|\nabla |\omega^\theta|^2|^2 +
  \frac{|\omega^\theta|^4}{r^2}\big)dx\\\nonumber
&&\leq \big\|\frac{u^r}{r}\big\|_{L^\infty}\int(\omega^\theta)^4dx
  + \|\Pi\|_{L^\infty}\|B^\theta\|_{L^\infty}\|\partial_z|\omega^\theta|^2\|_{L^2}
  \|\omega^\theta\|_{L^2}\\\nonumber
&&\leq \big\|\frac{u^r}{r}\big\|_{L^\infty}\int(\omega^\theta)^4dx
  +  \frac{1}{2}\|\Pi\|_{L^\infty}^2
  \|B^\theta\|_{L^\infty}^2\|\omega^\theta\|_{L^2}^2
  + \frac{1}{2}\|\partial_z|\omega^\theta|\|_{L^2}^2.
\end{eqnarray}
Using the \textit{a priori} estimate in \eqref{3-3} and Lemma
\ref{C-Z}, one has
\begin{eqnarray}\nonumber
\||\omega^\theta|^2\|_{L^\infty([0, t], L^2(\mathbb{R}^3))}^2 +
\|\nabla|\omega^\theta|^2\|_{L^2([0, t], L^2(\mathbb{R}^3))}^2
\lesssim e^{t^{\frac{5}{4}}}.
\end{eqnarray}
By Sobolev imbedding inequality, one has
\begin{eqnarray}\nonumber
\|\omega^\theta\|_{L^\infty([0, t], L^4(\mathbb{R}^3))} +
\|\omega^\theta\|_{L^4([0, t], L^{12}(\mathbb{R}^3))} \lesssim
e^{t^{\frac{5}{4}}}.
\end{eqnarray}
On the other hand, by Sobolev imbedding, one also has
\begin{eqnarray}\nonumber
\|{\bf u}\|_{L^\infty([0, t], L^\infty(\mathbb{R}^3))} \lesssim
\|{\bf u}\|_{L^\infty([0, t], L^2(\mathbb{R}^3))} +
\|\omega^\theta\|_{L^\infty([0, t], L^4(\mathbb{R}^3))} \lesssim
e^{t^{\frac{5}{4}}}.
\end{eqnarray}
Hence, we have
\begin{equation}\nonumber
\|(\nabla\times {\bf u})\times {\bf u}\|_{L^4([0, t],
L^{12}(\mathbb{R}^3))} \lesssim e^{t^{\frac{5}{4}}}.
\end{equation}

Write
\begin{equation}\nonumber
\nabla\times {\bf u} = e^{t\Delta}\nabla\times {\bf u}_0 -
\int_0^te^{(t - s)\Delta}\big(\nabla\times[(\nabla\times {\bf
u})\times{\bf u}] + \partial_z(\Pi B^\theta e_\theta)\big)ds.
\end{equation}
A standard parabolic estimate gives that
\begin{equation}\nonumber
\|\nabla\nabla\times {\bf u}\|_{L^4([0, t], L^{12}(\mathbb{R}^3))}
\lesssim  e^{t^{\frac{5}{4}}}.
\end{equation}
By Sobolev imbedding, we have
\begin{equation}\label{1-2}
\|\nabla {\bf u}\|_{L^4([0, t], L^\infty(\mathbb{R}^3))} \lesssim
e^{t^{\frac{5}{4}}}.
\end{equation}

Now let us derive the  $L^1([0, T], {\rm Lip}(\mathbb{R}^3))$
estimate of ${\bf B}$. We first write
\begin{equation}\nonumber
\partial_t{\bf B} + {\bf u}\cdot\nabla{\bf B} = \frac{u^r}{r}{\bf
B}.
\end{equation}
Applying $\nabla$, one has
\begin{equation}\nonumber
\partial_t\nabla{\bf B} + {\bf u}\cdot\nabla\nabla{\bf B} = - \nabla{\bf u}\cdot\nabla{\bf B}
+  \frac{u^r}{r}\nabla{\bf B} +  \nabla u^r\Pi {\bf e}_\theta +
(\nabla \frac{1}{r})u^r{\bf B}.
\end{equation}
Note that
\begin{equation}\nonumber
(\nabla \frac{1}{r})u^r{\bf B} = - \frac{u^r}{r}\Pi {\bf e}_r,
\end{equation}
one has
\begin{eqnarray}\nonumber
\|\nabla{\bf B}(t, \cdot)\|_{L^\infty} &\lesssim& \|\nabla{\bf
  B}_0\|_{L^\infty} + \int_0^t\big(\|\nabla{\bf u}\|_{L^\infty} +
  \big\|\frac{u^r}{r}\big\|_{L^\infty}\big)\|\nabla{\bf B}(s,
  \cdot)\|_{L^\infty}ds\\\nonumber
&&\quad +\ \int_0^t\big(
  \|\nabla{\bf u}\|_{L^\infty} + \big\|\frac{u^r}{r}
  \big\|_{L^\infty}\big)\|\Pi(s, \cdot)\|_{L^\infty}ds.
\end{eqnarray}
We can use \eqref{3-3}, \eqref{1-2}, Lemma \ref{C-Z} and Gronwall's
inequality to estimate that
\begin{eqnarray}\label{1-3}
\|\nabla{\bf B}(t, \cdot)\|_{L^\infty} \lesssim
e^{e^{t^{\frac{5}{4}}}}.
\end{eqnarray}

The \textit{a priori} estimates \eqref{1-2} and \eqref{1-3} are
enough for the global regularity of the ideal MHD equations
\eqref{iMHD}. Indeed, applying the standard $H^2$ energy estimate,
one has
\begin{eqnarray}\nonumber
&&\frac{1}{2}\frac{d}{dt}\big(\|\nabla^2{\bf u}(t, \cdot)\|_{L^2}^2
  + \|\nabla^2{\bf B}(t, \cdot)\|_{L^2}^2\big) + \|\nabla^3{\bf u}
  (t, \cdot)\|_{L^2}^2\\\nonumber
&&\lesssim \int\big(- \nabla^2{\bf u}\nabla^2({\bf u}\cdot\nabla
  {\bf u}) + \nabla^2{\bf u}\nabla^2({\bf B}\cdot\nabla {\bf
  B})\big) dx\\\nonumber
&&\quad +\ \int\big(- \nabla^2{\bf B}\nabla^2({\bf u}\cdot\nabla
  {\bf B}) + \nabla^2{\bf B}\nabla^2({\bf B}\cdot\nabla {\bf
  u})\big) dx\\\nonumber
&&\lesssim \frac{1}{2}\|\nabla^3{\bf u}(t, \cdot)\|_{L^2}^2
  + \|{\bf u}(t, \cdot)\|_{L^\infty}^2\|\nabla^2{\bf u}
  (t, \cdot)\|_{L^2}^2 + \|{\bf B}(t, \cdot)\|_{L^\infty}^2\|\nabla^2{\bf B}
  (t, \cdot)\|_{L^2}^2\\\nonumber
&&\quad +\ \big(\|\nabla{\bf u}(t, \cdot)\|_{L^\infty} +
  \|\nabla{\bf B}(t, \cdot)\|_{L^\infty}\big)\big(\|\nabla^2{\bf
  B}(t, \cdot)\|_{L^2}^2 + \|\nabla^2{\bf u}
  (t, \cdot)\|_{L^2}^2\big).
\end{eqnarray}
Here we also used the Gagliardo-Nirenberg's inequality $\|\nabla
f\|_{L^4}^2 \lesssim \|f\|_{L^\infty}\|\nabla^2f\|_{L^2}$, the
integration by parts and $\int\nabla^2{\bf B}({\bf u}\cdot\nabla)
  \nabla^2{\bf B}dx = 0$. Consequently, one has
\begin{equation}\nonumber
\begin{cases}
\|\nabla^2 {\bf u}(t, \cdot)\|_{L^2} + \|\nabla^2 {\bf B}(t,
\cdot)\|_{L^2} \lesssim e^{e^{e^{t^{\frac{5}{4}}}}},\\[-4mm]\\
\int_0^t\|\nabla^3{\bf u}\|_{L^2}^2dt \lesssim
e^{e^{e^{t^{\frac{5}{4}}}}},
\end{cases}
\quad \forall\ t \geq 0.
\end{equation}
We finished the proof of Theorem \ref{iMHD}.
\end{proof}

\begin{rem}
The proof of Lemma \ref{C-Z} can also be proved by using the following Biot-Savart law (see
\cite{ShirotaYanagisawa94}):
\begin{equation}\nonumber
|u^r(t, x)| \lesssim \int_{|y - x| \leq 4r}\frac{|\omega^\theta(t,
y)|}{|x - y|^2}dy + r\int_{|y - x| \geq 4r}\frac{|\omega^\theta(t,
y)|}{|x - y|^3}dy,
\end{equation}
which, by
Young's inequality, gives that
\begin{eqnarray}\nonumber
|u^r(t, r, z)| \lesssim r\frac{1}{|x|^2}\ast\Omega \leq
r\big\|\frac{1}{|x|^2}\big\|_{L^{\frac{3}{2},
\infty}}\|\Omega\|_{L^{3,1}} .
\end{eqnarray}
Here $L^{p, q}$ denotes the usual Lorentz norm. Then using the real interpolation
and Sobolev imbedding, one has
\begin{eqnarray}\label{8}
\big|\frac{u^r(t, r, z)}{r}\big| \lesssim \|\Omega\|_{L^{3, 1}} \leq
\|\Omega\|_{L^2}^{\frac{1}{2}}\|\Omega\|_{L^{\frac{1}{2}}}.
\end{eqnarray}
\end{rem}

\section{Proofs of Theorem \ref{thm-RMHD}}

In this section we prove Theorem \ref{thm-RMHD}.

\begin{proof}[Proof of Theorem \ref{thm-RMHD}]
Similarly as in obtaining \eqref{3-1}, one has
\begin{eqnarray}\nonumber
\frac{1}{2}\frac{d}{dt}\|\Omega\|_{L^2}^2 +
  \|\nabla\Omega\|_{L^2}^2
  &\lesssim& \big|\int\Omega\partial_z\Pi^2dx\big|\\\nonumber
&\leq& \|\Pi\|_{L^\infty}^{\frac{1}{3}}\|\Pi\|_{L^{\frac{10}{3}}}
  ^{\frac{5}{3}}\|\partial_z\Omega\|_{L^2}\\\nonumber &\lesssim&
  \|\Pi\|_{L^\infty}^{\frac{1}{3}}\|\Pi\|_{L^2}^{\frac{2}{3}}
  \|\nabla\Pi\|_{L^2}\|\partial_z\Omega\|_{L^2}.
\end{eqnarray}
Consequently, one has
\begin{eqnarray}\label{2}
\frac{d}{dt}\|\Omega\|_{L^2}^2 + \|\nabla\Omega \|_{L^2}^2 \lesssim
\|\Pi\|_{L^\infty}^{\frac{2}{3}}\|\Pi\|_{L^2}^{\frac{4}{3}}\|\nabla\Pi\|_{L^2}^2.
\end{eqnarray}

Applying a similar argument to $\Pi$ equation in \eqref{Pi-eqn-2},
one has
\begin{eqnarray}\label{3}
\frac{d}{dt}\|\Pi\|_{L^2}^2 + \|\nabla\Pi\|_{L^2}^2 \leq 0.
\end{eqnarray}
Clearly, the combination of \eqref{2}, \eqref{3} and the maximum
estimate in Proposition \ref{prop-maxi} gives the following
\textit{a priori} estimate
\begin{equation}\label{4}
\|\Pi(t, \cdot)\|_{L^2} + \|\Omega(t, \cdot)\|_{L^2} \lesssim 1\ \
(\forall\ t \geq 0),\quad \int_0^\infty\big(\|\nabla\Pi\|_{L^2}^2 +
\|\nabla\Omega\|_{L^2}^2\big)dt \lesssim 1.
\end{equation}

Now let us come back to the equation of $\omega^\theta$ in
\eqref{vorticity-axi}. Applying the standard energy estimate, one
has
\begin{eqnarray}\label{1-1}
&&\frac{1}{2}\frac{d}{dt}\int|\omega^\theta|^2dx +
  \int\big(|\nabla \omega^\theta|^2 +
  \frac{|\omega^\theta|^2}{r^2}\big)dx\\\nonumber
&&= \int\frac{u^r(\omega^\theta)^2}{r}dx  - \int
  \frac{\partial_z(B^\theta)^2}{r}\omega^\theta dx.
\end{eqnarray}
Using Sobolev imbedding theorem and interpolation, one has
\begin{eqnarray}\nonumber
\Big|\int\frac{u^r(\omega^\theta)^2}{r}dx\Big| &\leq&
  \|u^r\|_{L^2}\|\Omega\|_{L^6}\|\omega^\theta\|_{L^3}\\\nonumber
&\lesssim& \|u^r\|_{L^2}\|\nabla\Omega\|_{L^2}\|\omega^\theta
  \|_{L^2}^{\frac{1}{2}}\|\nabla\omega^\theta\|_{L^2}^{\frac{1}{2}}\\\nonumber
&\lesssim& \|u^r\|_{L^2}^2\|\nabla\Omega\|_{L^2}^2 + \|\omega^\theta
  \|_{L^2}^2 + \frac{1}{4}\|\nabla\omega^\theta\|_{L^2}^2.
\end{eqnarray}
On the other hand, it is clear that one also has
\begin{eqnarray}\nonumber
\Big|\int\frac{\partial_z(B^\theta)^2}{r}\omega^\theta
  dx\Big| &\leq& \|\Pi\|_{L^\infty}\|B^\theta\|_{L^2}
  \|\partial_z\omega^\theta\|_{L^2}\\\nonumber
&\leq& \|\Pi\|_{L^\infty}^2\|B^\theta\|_{L^2}^2 +
  \frac{1}{4}\|\partial_z\omega^\theta\|_{L^2}.
\end{eqnarray}
Using the \textit{a priori} estimate \eqref{4}, the basic energy law
\eqref{EnergyL} and Proposition \ref{prop-maxi}, we have
\begin{eqnarray}\label{5}
\|\nabla\times {\bf u}(t, \cdot)\|_{L^2} \lesssim 1\ \ (\forall\ t
\geq 0),\quad \int_0^\infty\|\nabla(\nabla\times{\bf
u})\|_{L^2}^2\big)dt \lesssim 1.
\end{eqnarray}

The \textit{a priori} estimate \eqref{5} is enough to get the global
regularity of the resistive MHD \eqref{RMHD}. Indeed, by using the
equation of ${\bf B}$, one can easily verifies that $\nabla{\bf B}$
also satisfies \eqref{5}. We have finished the proof of Theorem
\ref{thm-RMHD}.
\end{proof}

\section{Appendix}

In this appendix we first prove that $w(y) = r^\alpha$ is a
$\mathcal{A}_p$ for Riesz operator in $\mathbb{R}^5$ under $-4 <
\alpha < 4p\big(1 - \frac{1}{p}\big)$. The case of $\alpha = - 2$
has been studied in \cite{HLL}.
\begin{lem}[$\mathcal{A}_p$ Weight]\label{A-p}
Let $1 < p < \infty$ and $w(y) = r^\alpha$, $y \in \mathbb{R}^5$.
Then $w(x)$ is in $\mathcal{A}_p$ class if $-4 < \alpha < 4p\big(1 -
\frac{1}{p}\big)$.
\end{lem}
\begin{proof}
Recall that a real valued non-negative function $w(x)$ is said to be
in $\mathcal{A}_p( \mathbb{R}^n)$ class  if it satisfies
\begin{equation}\nonumber
\sup_{B \subset \mathbb{R}^n
}\Big(\frac{1}{|B|}\int_Bw(x)dx\Big)\Big(\frac{1}{|B|}\int_Bw(x)^{-
\frac{q}{p}}dx\Big)^{\frac{p}{q}} < \infty.
\end{equation}
Here $p$ and $q$ are conjugate indices with $1 < p < \infty$.

 For
any ball $B \subset \mathbb{R}^5$, denote $B = B(y_0, R)$. It is
easy to see that if $r_0 > 2R$, one has $r \simeq r_0$ for any $x
\in B$. Consequently, for any $\alpha \in \mathbb{R}$, one has
\begin{eqnarray}\nonumber
&&\Big(\frac{1}{|B|}\int_Bw(x)dx\Big)\Big(\frac{1}{|B|}\int_Bw(x)^{-
  \frac{q}{p}}dx\Big)^{\frac{p}{q}}\\\nonumber
&&\lesssim \Big(\frac{1}{|B|}\int_Br_0^\alpha dx\Big)
  \Big(\frac{1}{|B|}\int_Br_0^{- \frac{q}{p}
  \alpha}dx\Big)^{\frac{p}{q}} \lesssim 1.
\end{eqnarray}
On the other hand, if $r_0 \leq 2R$, then for $\alpha + 3 > -1$ and
$- \frac{q\alpha}{p} + 3 > -1$, one has
\begin{eqnarray}\nonumber
&&\Big(\frac{1}{|B|}\int_Bw(x)dx\Big)\Big(\frac{1}{|B|}\int_Bw(x)^{-
  \frac{q}{p}}dx\Big)^{\frac{p}{q}}\\\nonumber
&&\lesssim \Big(\frac{1}{R^5}\int_{r_0 - R}^{r_0 + R}dz\int_{0}^{
  3R}r^{\alpha + 3}dr\Big)\Big(\frac{1}{R^5}\int_{r_0 - R}^{r_0 +
  R}dz\int_{0}^{3R}r^{- \frac{q\alpha}{p} + 3}dr\Big)^{\frac{p}{q}}\\\nonumber
&&\lesssim R^\alpha R^{-\alpha} = 1.
\end{eqnarray}
Noting that the condition on $\alpha$ is $- 4 < \alpha < 4p\big(1 -
\frac{1}{p}\big)$, we in fact have completed the proof of the lemma.
\end{proof}

\section*{Acknowledgement}
Part of this work was done when Zhen Lei was visiting the Courant
Institute of New York University and the Department of Mathematics
of Harvard University during 2012. He would like to thank professor
FangHua Lin, professor Horng-Tzer Yau and professor Shing-Tung Yau
for their interest in this work. The author was also supported by
NSFC (grant No.11171072, 11121101 and 11222107), FANEDD, Innovation Program of Shanghai
Municipal Education Commission (grant No.12ZZ012), Shanghai Talent Development Fund and SGST
09DZ2272900.

\end{document}